\documentclass[10pt]{amsart}

\usepackage{amssymb,amsthm,amsmath}
\usepackage{enumerate}
\usepackage{caption}
\usepackage{float}
\usepackage{graphicx}

\newcommand{\E}{\mathbb{E}}

\newcommand{\C}{\mathbb{C}}

\usepackage[paper=a4paper, left=1.2in, right=1.2in, top=1.2in, bottom=1.5in]{geometry}
\newtheorem{theorem}{Theorem}
\newtheorem{lemma}[theorem]{Lemma}

\theoremstyle{remark}
\newtheorem{remark}[theorem]{Remark}

\theoremstyle{definition}

\begin{document}

\title{Comparison of moments of Rademacher chaoses}

\author{Paata Ivanisvili}
\address{Princeton University; University of California, Irvine} \email{paatai@math.princeton.edu \textrm{(P. Ivanisvili)}}

\author{Tomasz Tkocz}
\address{Carnegie Mellon University}
\email{ttkocz@math.cmu.edu \textrm{(T. Tkocz)}}

\begin{abstract}
We show that complex hypercontractivity gives better constants than real hypercontractivity in comparison inequalities for (low) moments of Rademacher chaoses (homogeneous polynomials on the discrete cube).
\end{abstract}

\maketitle

{\footnotesize
\noindent {\em 2010 Mathematics Subject Classification.} Primary 60E15; Secondary 42C10.

\noindent {\em Key words.} Rademacher chaos, moment comparison, hypercontractivity, Hamming cube, Markov--Nikolskii type inequality
}

\section*{Introduction}

A Rademacher chaos $h$ of order (degree) $d$ is a $d$-homogeneous polynomial on the discrete cube $\{-1,1\}^n$ for some $n \geq d$, that is a function of the form $h(x) = \sum_{1 \leq i_1<\ldots<i_d \leq n} a_{i_1,\ldots,i_d}x_{i_1}\cdot\ldots\cdot x_{i_d}$, $x = (x_1,\ldots,x_n) \in \{-1,1\}^n$, for some, say complex coefficients $a_{i_1,\ldots,i_d}$. For $p > 0$, denote by $\|f\|_p$ the $p$-th moment $(\E |f|^p)^{1/p}$ of a function $f:\{-1,1\}^n\to \C$, with the expectation taken against the uniform probability measure on $\{-1,1\}^n$. Let $1 \leq p \leq q$. We are interested in moment comparison inequalities: $\|h\|_{q} \leq C_{p,q,d}\|h\|_{p}$, true for any Rademacher chaos $h$ of degree $d$ with constants $C_{p,q,d}$ dependent only on $p, q$ and $d$ (so independent of $n$ and the coefficients $a_{i_1,\ldots,i_d}$ of~$h$). When $d=1$, these are the Khinchin inequalities and sharp values of the constants $C_{p,q,1}$ are known in many cases (see for instance \cite{NO} for a recent result and further references).

One way of effortlessly obtaining such comparison inequalities is by \emph{real hypercontractivity}, which for $1 \leq p \leq q$ gives $C_{p,q,d} = \left(\frac{q-1}{p-1}\right)^{d/2}$ and
$C_{p,q,d} = e^{(2/p-2/q)d}$, when additionally $q \leq 2$  (see for example Theorem 5.10 in \cite{Jan} and Theorems 9.21, 9.22 in \cite{O'D}). To the best of our knowledge, these are in fact the best known values of constants $C_{p,q,d}$ (except for $p=2$ and $q$ being an even integer, where combinatorial arguments give slightly better results -- see \cite{Bon70} and Exercise 9.38 in \cite{O'D}). The constant $\left(\frac{q-1}{p-1}\right)^{d/2}$ is moreover asymptotically sharp as $d$ goes to infinity with $2 < p < q$ fixed (see \cite{Student}), in the sense that one cannot replace it by  $C^{d/2}$ with $C <\frac{q-1}{p-1}$ as $d\to\infty$. 

The purpose of this note is to further improve the constants for low moments ($p \leq 2$). The key is an observation that complex hypercontractivity yields better comparison between $p$-th and $q$-th moments than real hypercontractivity for $p< 2 < q$, which is the statement of the next theorem.

\begin{theorem}\label{thm:moments-base}
Let $1 < p \leq 2 \leq q.$ Let $h:\{-1,1\}^{n} \to \mathbb{C}$ be a $d$-homogeneous polynomial. We have, 
\begin{align}
&\|h\|_{q} \leq \max\left\{(q-1)^{d/2}, \frac{1}{(p-1)^{d/2}} \right\}\|h\|_{p}. \label{eq:Lp-Lq-base}
\end{align}
\end{theorem}

Our main result is obtained by the usual interpolation of moments, which can be viewed as a self-improvement of \eqref{eq:Lp-Lq-base}. 

\begin{theorem}\label{thm:moments-selfimp}
Let $1 \leq p \leq q.$ Let $h:\{-1,1\}^{n} \to \mathbb{C}$ be a $d$-homogeneous polynomial. We have,
\begin{equation}\label{eq:mainres}
\|h\|_q \leq C_{p,q,d}\|h\|_p,
\end{equation}
with $C_{p,q,d} = \begin{cases} 
\exp\left\{\left(\frac{1}{p}-\frac{1}{q}\right)d\right\}, & \text{ if } 1 \leq p \leq q \leq 2, \\
(q-1)^{\frac{q-p}{p(q-2)}\frac{d}{2}}, & \text{ if } 1 \leq p \leq 2 \leq q \text{ and } \frac{1}{p}+\frac{1}{q} > 1,\\
(q-1)^{\frac{d}{2}}, & \text{ if } 1 \leq p \leq 2 \leq q \text{ and } \frac{1}{p}+\frac{1}{q} \leq 1,\quad \text{asymp. sharp as} \; d \to \infty,\\
\left(\frac{q-1}{p-1}\right)^{\frac{d}{2}}, & \text{ if } 2 < p \leq q, \qquad \qquad \qquad \qquad \quad  \; \text{asymp. sharp as} \; d \to \infty.
\end{cases}$
\end{theorem}

\begin{remark}
The constant in the first case clearly improves (by the factor of $2$ in the exponent) on the constant $e^{(2/p-2/q)d}$ obtained from real hypercontractivity. It can be checked that the constant in the second case improves on the constant $(p-1)^{-d/2}$ given by \eqref{eq:Lp-Lq-base}. The constants in the third and fourth cases are directly obtained from the complex and real hypercontractivity, respectively (we stated them for completeness). We also mention in passing that \eqref{eq:mainres} can be seen as a discrete-cube analogue of the classical Nikolskii type inequalities for polynomials (with the constant in the first case being of a similar form -- see for instance Theorem 2.6 in \cite{DeVore} and \cite{Nik}) .
\end{remark}

\begin{remark}
The constants in the third and fourth cases are asymptotically sharp as $d \to \infty$. Indeed, sharpness follows from the example of Hermite polynomials and the application of the central limit theorem. The asymptotics of $L_{p}$ norms of Hermite polynomials are computed in  \cite{Student}.  
\end{remark}

\begin{remark}
In the case $p=1$ and $q=2$, we obtain $C_{1,2,d} = e^{d/2}$. It is widely believed that the best possible $C_{1,2,d}$ should be $2^{d/2}$ (which is attained for $h(x) = (x_1+x_2)(x_3+x_4)\cdot\ldots\cdot(x_{2d-1}+x_{2d})$). For example, Pe\l czy\'nski's conjecture states that $C_{1,2,2} = 2$ (for chaoses with coefficients in arbitrary normed spaces, see \cite{Ole}). 
\end{remark}

\begin{remark}
It remains an open problem to determine the sharp values of the constants $C_{p,q,d}$ (even asymptotically, with $d \to \infty$, except for the case $2\leq p \leq q$, and $1\leq p\leq 2 \leq q$ with $\frac{1}{p}+\frac{1}{q}\leq 1$).
\end{remark}

\begin{remark}
Based on arguments from \cite{KW} (see Lemma 6.4.1), it is possible to extend the moment comparison from Theorem \ref{thm:moments-selfimp} to all polynomials of degree (at most) $d$ (that is, to not necessarily homogeneous polynomials). However, the constants we obtain this way are perhaps far from optimal. 
\end{remark}

\section*{Complex hypercontractivity and proof of Theorem \ref{thm:moments-base}}

For $x=(x_{1}, \ldots, x_{n})\in \{-1,1\}^{n}$ and $S\subseteq [n]:=\{1,2,\ldots, n\}$, we define the Walsh functions
$
w_{S}(x) = \prod_{j \in S}x_{j}.
$
When $S = \varnothing$, we set $w_{\varnothing}(x)=1$ for all $x \in \{-1,1\}^{n}$. These functions form an orthogonal basis $\{w_S, \ S \subset [n]\}$ in the space of all functions $f : \{-1,1\}^{n} \to \mathbb{C}$ and thus any such function has the Fourier--Walsh expansion 
\begin{align*}
f(x) = \sum_{S \in [n]} a_{S} w_{S}(x),
\end{align*}
where $a_{S} = \mathbb{E} f w_{S}$.  By $|S|$ we denote the cardinality of the set $S$. Take any $z \in \mathbb{C}$ and define the operator $T_{z}$ as follows 
\begin{align*}
T_{z} f(x) =  \sum_{S \in [n]} z^{|S|}a_{S} w_{S}(x).
\end{align*}
Real hypercontractivity tells us that for $1 < p < q$ and $z = \sqrt{\frac{q-1}{p-1}}$, the operator $T_z$ is a contraction from $L_p$ to $L_q$, that is $\|T_zf\|_q \leq \|f\|_p$ for all $f:\{-1,1\}^n\to \C$ (see for instance \cite{O'D}).

In what follows $q\geq 2\geq p \geq 1$. By the result of Weissler~\cite{WEISSLER} (see also Beckner~\cite{BECKNER} for dual exponents $p$ and $q$), for $t \in \mathbb{R}$, we have 
\begin{align}\label{wewe}
\|T_{it} f\|_{q}\leq \|f\|_{p} \quad \text{for all} \quad f:\{-1,1\}^{n} \to \mathbb{C}
\end{align}
if and only if 
\begin{align}\label{cc1}
|t| \leq \min\left\{\sqrt{p-1} , \frac{1}{\sqrt{q-1}} \right\}.
\end{align}
In particular, for any $d$-homogeneous polynomial $h:\{-1,1\}^{n} \to \mathbb{C}$, it yields
\begin{align*}
\min\left\{(p-1)^{d/2}, \frac{1}{(q-1)^{d/2}} \right\} \|h\|_{q}\leq \|h\|_{p}
\end{align*}
(because $T_zh = z^dh$), and this finishes the proof of Theorem~\ref{thm:moments-base}. 
\hfill$\square$

\begin{remark}
There is a conjecture of Weissler from \cite{WEISSLER} that for $z \in \mathbb{C}$, $|z|\leq 1$ we have 
\begin{align}\label{vei}
\|T_{z} f\|_{q}\leq \|f\|_{p} \quad \text{for all} \quad f:\{-1,1\}^{n} \to \mathbb{C}
\end{align}
if and only if 
\begin{align}\label{inf}
(q-2) (\Re\, wz)^{2} + |wz|^{2} \leq (p-2)(\Re w)^{2} + |w|^{2} \quad \text{for all} \quad w \in \mathbb{C}.
\end{align}
The conjecture is partially resolved, with the only case left open being $2<p<q<3$ and its dual, i.e., $3/2<p<q<2$. One cannot improve the bound in Theorem~\ref{thm:moments-base} and Theorem~\ref{thm:moments-selfimp} even if one uses (\ref{vei}) and (\ref{inf}) in its full generality instead of (\ref{wewe}) and  (\ref{cc1}), i.e., the particular case of (\ref{vei}), (\ref{inf}) when $z$ is purely imaginary. 
\end{remark}

\section*{Self improvement of \eqref{eq:Lp-Lq-base} and proof of Theorem \ref{thm:moments-selfimp}}

\begin{figure}[ht]
\centering
\includegraphics[scale=1]{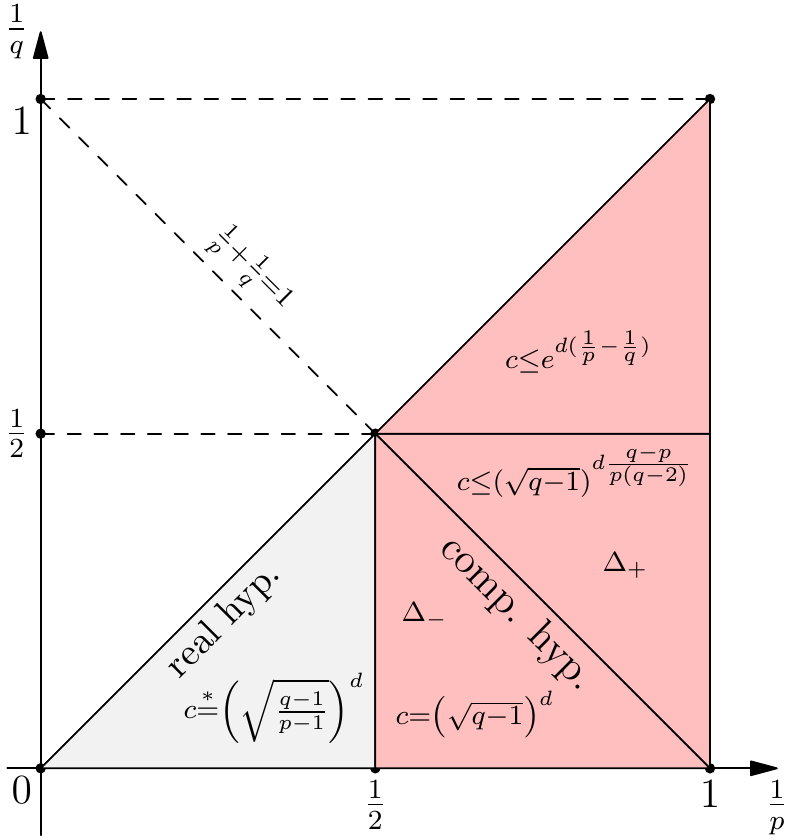}
\caption{Inequality $\|h\|_{q}\leq c \|h\|_{p}$ for $q\geq p \geq 1$;}

\label{fig:dom}
\end{figure}


Fix a $d$-homogeneous polynomial $h$ and consider the function $\psi(s) = \frac{1}{d}\log\|h\|_{1/s}$ on $(0,1]$, which is nonincreasing and convex (by H\"older's inequality). We set $s = \frac{1}{p}$ and $t = \frac{1}{q}$. 
Define the region $R_{s,t} = \{(x,y), \ 0 < y \leq x, x \leq s, y \leq t\}$. By convexity, the slopes of $\psi$ are nondecreasing, thus
\begin{equation}\label{eq:interpol}
\frac{1}{d}\frac{1}{1/q-1/p}\log\frac{\|h\|_q}{\|h\|_p} = \frac{\psi(t)-\psi(s)}{t-s} \geq \sup_{(x,y) \in R_{s,t}} \frac{\psi(y) - \psi(x)}{y-x}.
\end{equation}
Define regions where we can use \eqref{eq:Lp-Lq-base}: $\Delta_- = \{(x,y), \ 0 < y \leq \frac{1}{2} \leq x < 1, x + y \leq 1\}$ and $\Delta_+ = \{(x,y), \ 0 < y \leq \frac{1}{2} \leq x < 1, x + y > 1\}$. It follows from \eqref{eq:Lp-Lq-base} that,
\begin{align*}
\psi(y) - \psi(x) &\leq \frac{1}{2}\log(y^{-1}-1) & \text{on } \Delta_-,\\
\psi(y) - \psi(x) &\leq -\frac{1}{2}\log(x^{-1}-1) & \text{on } \Delta_+.
\end{align*}
Therefore,
\begin{equation}\label{eq:interpol2}
\sup_{R_{s,t}} \frac{\psi(y) - \psi(x)}{y-x} \geq \frac{1}{2}\max\left\{\sup_{R_{s,t}\cap \Delta_-} \frac{\log(y^{-1}-1)}{y-x}, \sup_{R_{s,t}\cap \Delta_+}\frac{-\log(x^{-1}-1)}{y-x}\right\}.
\end{equation}
To compute the right hand side, we shall need the following elementary fact.

\begin{lemma}\label{lm:monot}
For every $\frac{1}{2}\leq s \leq 1$, the function $\beta_s(u) = \frac{\log(u^{-1}-1)}{u-s}$ is increasing on $(0,s)$.
\end{lemma}
\begin{proof}
We have, $(u-s)^2\beta_s'(u) =\frac{s-u}{u(1-u)} - \log(u^{-1}-1)$, which is positive for $u \in (0,s)$ if and only if $s > u + u(1-u)\log(u^{-1}-1)$. The derivative of the right hand side is $(1-2u)\log(u^{-1}-1)$, which is positive, so it suffices to check that $s > s + s(1-s)\log(s^{-1}-1)$, which is clearly true for every $\frac{1}{2} < s < 1$.
\end{proof}

\noindent
In particular, since the function $\beta_{1/2}(u) = \frac{\log(u^{-1}-1)}{u-1/2}$ satisfies $\beta_{1/2}(1-u) = \beta_{1/2}(u)$, it is symmetric about $u = \frac12$, it increases on $(0,\frac12)$ and it decreases on $(\frac12,1)$. Moreover, $\lim_{u\to \frac12}\beta_{1/2}(u) = -4$ and $\beta_u(1-u) = \frac12\beta_{1/2}(u)$.

\paragraph{\emph{Case 1.}} 
$1 \leq p \leq q \leq 2$, that is $\frac{1}{2}\leq t \leq s \leq 1$. We have,
\begin{align*}
\sup_{R_{s,t}\cap \Delta_-} \frac{\log(y^{-1}-1)}{y-x} = \sup_{\substack{\frac12 \leq x \leq s \\ y \leq 1-x}} \beta_x(y) = \sup_{\frac12 \leq x \leq s} \beta_x(1-x)&=  \sup_{\frac12 \leq x \leq s} \frac{1}{2}\beta_{1/2}(x) = -2.
\end{align*}
Using the evident monotonicity in $y$,
\[
\sup_{R_{s,t}\cap \Delta_+} \frac{-\log(x^{-1}-1)}{y-x} = \sup_{\substack{\frac12 \leq x \leq s \\ 1-x < y \leq \frac12}} \frac{-\log(x^{-1}-1)}{y-x} = \sup_{\frac12 \leq x \leq s} \frac12\beta_{1/2}(x) = -2.
\]
Therefore, by \eqref{eq:interpol2}, $\sup_{R_{s,t}} \frac{\psi(y) - \psi(x)}{y-x} \geq -1$, so \eqref{eq:interpol} yields
\[
\|h\|_q/\|h\|_p \leq \exp\left\{\left(\frac{1}{p}-\frac{1}{q}\right)d\right\}.
\]

\paragraph{\emph{Case 2.}}
$1 \leq p \leq 2 \leq q \text{ and } \frac{1}{p}+\frac{1}{q} > 1$, that is $t \leq \frac12 \leq s \leq 1$ and $s+t > 1$. We have,
\begin{align*}
\sup_{R_{s,t}\cap \Delta_-} \frac{\log(y^{-1}-1)}{y-x} = \sup_{\substack{\frac12 \leq x \leq s \\ y \leq 1-x, y \leq t}} \beta_x(y) = \max\left\{\sup_{\substack{1-t \leq x \leq s}} \beta_x(1-x), \sup_{\substack{\frac12 \leq x \leq 1-t}} \beta_x(t)\right\}.
\end{align*}
As before, 
\[
\sup_{\substack{1-t \leq x \leq s}} \beta_x(1-x) = \sup_{\substack{1-t \leq x \leq s}} \frac12\beta_{1/2}(x) = \frac12\beta_{1/2}(1-t) = \frac12\beta_{1/2}(t).\]
Moreover, by the evident monotonicity in $x$,
\[
\sup_{\substack{\frac12 \leq x \leq 1-t}} \beta_x(t) = \sup_{\substack{\frac12 \leq x \leq 1-t}} \frac{\log(t^{-1}-1)}{t-x} = \frac{\log(t^{-1}-1)}{2t-1} = \frac12\beta_{1/2}(t).
\]
Therefore,
$
\sup_{R_{s,t}\cap \Delta_-} \frac{\log(y^{-1}-1)}{y-x} = \frac{1}{2}\beta_{1/2}(t).
$
A similar computation shows that the supremum over the region $\Delta_+$ also gives $\frac{1}{2}\beta_{1/2}(t).$ Thus, $\sup_{R_{s,t}} \frac{\psi(y) - \psi(x)}{y-x} \geq \frac{1}{4}\beta_{1/2}(t) = \frac{\log(t^{-1}-1)}{2(2t-1)}$, so \eqref{eq:interpol} yields
\[
\|h\|_q/\|h\|_p \leq \exp\left\{\frac{\log(q-1)}{2(2/q-1)}\left(\frac{1}{q}-\frac{1}{p}\right)d\right\} = (q-1)^{\frac{q-p}{p(q-2)}\frac{d}{2}}.
\]

\paragraph{\emph{Case 3.}}
$1 \leq p \leq 2 \leq q \text{ and } \frac{1}{p}+\frac{1}{q} \leq 1$, that is $t \leq \frac12 \leq s \leq 1$ and $s+t \leq 1$. Here, $R_{s,t} \cap \Delta_+ = \varnothing$ and it can be checked as in the previous cases that the right hand side of \eqref{eq:interpol2} gives $\frac{1}{2}\frac{\log(t^{-1}-1)}{t-s}$, which does not improve on \eqref{eq:Lp-Lq-base}.

\paragraph{\emph{Case 4.}}
$2 < p \leq q$, that is $t \leq s < \frac{1}{2}$. Here, $R_{s,t} \cap \Delta_+ = R_{s,t} \cap \Delta_- = \varnothing$. By real hypercontractivity, $\sup_{(x,y) \in R_{s,t}} \frac{\psi(y) - \psi(x)}{y-x} \geq \sup_{(x,y) \in R_{s,t}} \frac{1}{2}\frac{\log(y^{-1}-1)-\log(x^{-1}-1)}{y-x}$ and, by convexity, this equals $\frac{1}{2}\frac{\log(t^{-1}-1)-\log(s^{-1}-1)}{t-s}$ (no self-improvement).
\hfill$\square$

\begin{remark}
Taking into account real hypercontractivity, a priori, the right hand side of \eqref{eq:interpol2} could have been replaced by
\[
\frac{1}{2}\max\left\{\sup_{R_{s,t}\cap \Delta_-} \frac{\log(y^{-1}-1)}{y-x}, \sup_{R_{s,t}\cap \Delta_+}\frac{-\log(x^{-1}-1)}{y-x}, \sup_{R_{s,t}\cap \Delta_0} \frac{\log(y^{-1}-1)-\log(x^{-1}-1))}{y-x}\right\},
\]
where $\Delta_0 = \{(x,y), \ 0 < y < x < \frac{1}{2}\} \cup \{(x,y), \ \frac12 < y < x < 1\}$.
It can be checked that in each Case 1 -- 3, this does not lead to further improvements (in other words, there is no loss in our argument being restricted to the regions $\Delta_{\pm}$).
\end{remark}

\section*{Acknowledgments}

We would like to thank R.~O'Donnell and K.~Oleszkiewicz for their comments regarding an early version of this manuscript.

This material is partially based upon work supported by the NSF under Grant No. 1440140, while the authors were in residence at the MSRI in Berkeley, California, during the fall semester of 2017.


\begin{thebibliography}{99}

\bibitem{BECKNER}
Beckner, W.,
Inequalities in Fourier analysis, \emph{Ann. of Math.} (2) 102 (1975), no. 1, 159--182.

\bibitem{Bon70}
Bonami, A., \'Etude des coefficients Fourier des fonctions de $L_p(G)$. \emph{Ann. Inst. Fourier (Grenoble)} 20 1970 fasc. 2, 335--402.

\bibitem{DeVore}
DeVore, R., Lorentz, G., Constructive approximation. 
\emph{Springer-Verlag, Berlin}, 1993.

\bibitem{Jan}
Janson, S.,
Gaussian Hilbert Spaces, \emph{Cambridge University Press}, 1997.

\bibitem{KW} Kwapie\'n, S., Woyczy\'nski, W.,
Random series and stochastic integrals: single and multiple. 
Probability and its Applications. \emph{Birkh\"auser Boston, Inc., Boston, MA}, 1992.  

\bibitem{Student}
Larsson-Cohn, L., $L_p$-norms of Hermite polynomials and an extremal problem on Wiener chaos. 
\emph{Ark. Mat.} 40 (2002), no. 1, 133--144. 

\bibitem{NO}
Nayar, P., Oleszkiewicz, K.,
Khinchine type inequalities with optimal constants via ultra log-concavity.
\emph{Positivity} 16 (2012), no. 2, 359--371. 

\bibitem{Nik}
Nikolskii, S. M.,
Inequalities for entire functions of finite degree and their application in the theory of differentiable functions of several variables. \emph{Trudy Mat. Inst. Steklov.}, v. 38 (1951), 244--278. 

\bibitem{O'D}
O'Donnell, R.,
Analysis of Boolean functions, \emph{Cambridge University Press}, 2014.

\bibitem{Ole}
Oleszkiewicz, K., Comparison of moments via Poincar\'e-type inequality.
\emph{Advances in stochastic inequalities (Atlanta, GA, 1997)}, 135--148, Contemp. Math., 234, \emph{Amer. Math. Soc., Providence, RI}, 1999.

\bibitem{WEISSLER}
Weissler, F.,
Two-point inequalities, the Hermite semigroup, and the Gauss-Weierstrass semigroup. 
\emph{J. Funct. Anal.} 32 (1979), no. 1, 102--121.

\end{thebibliography}
\end{document}